\documentclass[psamsfonts]{amsart}
\usepackage{amssymb}
\usepackage{graphicx,color}
\usepackage{amssymb,amsfonts,amsthm}
\usepackage[all,arc]{xy}
\usepackage{enumerate}
\usepackage{mathrsfs}

\newtheorem{thm}{Theorem}[section]

\newtheorem{lemma}[thm]{Lemma}

\theoremstyle{definition}

\newtheorem{exmp}[thm]{Example}

\theoremstyle{remark}

\makeatletter
\let\c@equation\c@thm
\makeatother
\numberwithin{equation}{section}

\title[Hypersurfaces of Prescribed Curvature]{Smooth Solutions to Asymptotic Plateau Type Problem in Hyperbolic Space}

\author{Zhenan Sui}

\address{Institute for Advanced Study in Mathematics of HIT, Harbin Institute of Technology, Harbin, China}
\email{sui.4@osu.edu}

\author{Wei Sun}

\address{Institute of Mathematical Sciences, ShanghaiTech University, Shanghai, China}
\email{sunwei@shanghaitech.edu.cn}

\begin{document}

\begin{abstract}
We investigate on the existence of smooth complete hypersurface with prescribed Weingarten curvature and asymptotic boundary at infinity in hyperbolic space under the assumption that there exists an asymptotic subsolution. We give an affirmative answer for the case $k = n$ when the asymptotic boundary $\Gamma$ bounds a uniformly convex domain, and for $k < n$ when $\Gamma$ bounds a disk, utilizing Pogorelov type interior second order estimate. Our result complements our previous work \cite{Sui2019, Sui-Sun}, and generalizes the asymptotic Plateau type problem to non-constant prescribed curvature case. 
\end{abstract}


\maketitle


\section {\large Introduction}

\vspace{4mm}

In this paper, we shall continue our study of finding smooth hypersurfaces to asymptotic Plateau type problem in hyperbolic space, which extends our previous work \cite{Sui2019, Sui-Sun}. As before, we take the half space model for hyperbolic space
\[\mathbb{H}^{n+1} = \{ (x, x_{n+1}) \in \mathbb{R}^{n+1} \big\vert x_{n+1} > 0\}, \]
endowed with the metric
\[ d s^2 = \frac{1}{x_{n+1}^2} \sum_{i = 1}^{n+1} d x_i^2. \]
Let $\psi$ be a given smooth positive function defined in $\mathbb{H}^{n + 1}$, and $\Gamma = \{\Gamma_1, \ldots, \Gamma_m\}$ be a disjoint collection of smooth closed $(n - 1)$ dimensional submanifolds at
\[ \partial_{\infty} \mathbb{H}^{n+1} := \mathbb{R}^n \times \{0\} \cong \mathbb{R}^n. \]
We want to seek a smooth complete connected admissible vertical graph $\Sigma = \{ (x, u(x) ) | x \in \Omega\}$ satisfying
\begin{equation} \label{eqn9}
\left\{ \begin{aligned}
f ( \kappa [ u ] ) =  & \, \sigma_k^{\frac{1}{k}} ( \kappa ) = \psi(x, u) \quad & \mbox{in} \,\, \Omega, \\
u  = &  \, 0 \quad & \mbox{on} \,\, \Gamma,
\end{aligned} \right.
\end{equation}
where $\kappa = (\kappa_1, \ldots, \kappa_n)$ are the hyperbolic principal curvatures of $\Sigma$ with respect to the upward normal,
\[\sigma_k (\kappa) = \sum_{1 \leq i_1 < \ldots < i_k \leq n} \kappa_{i_1} \cdots \kappa_{i_k} \]
is the $k$th-Weingarten curvature defined on $k$-th G\r arding's cone
\[\Gamma_k = \{ \kappa \in \mathbb{R}^n \vert \sigma_j(\kappa) > 0,\, j = 1, \ldots, k \}, \]
and $\Omega$ is the bounded domain enclosed by $\Gamma$ on $\mathbb{R}^n$.
$\Sigma$ is said to be admissible if $\kappa \in \Gamma_k$. When $k = 1, 2, n$, $\sigma_k(\kappa)$ is the so-called mean curvature, scalar curvature and Gauss curvature. In the special case $k = n$, we shall use ``strictly locally convex'' instead of ``admissible'' when $\kappa \in \Gamma_n$. Besides, we say that $\Sigma$ is locally convex if $\kappa \in \overline{\Gamma_n}$.

Asymptotic Plateau type problem \eqref{eqn9} is more difficult than common Dirichlet problems due to the singularity at $\Gamma$. Because of this, it resembles those problems on noncompact domains.
When $\psi = \sigma \in (0, 1)$ is a prescribed constant, Guan, Spruck, Szapiel and Xiao \cite{GS00, GSS09, GS10, GS11, GSX14} developed some effective technique for problem \eqref{eqn9}. They analyze the approximating
Dirichlet problem
\begin{equation} \label{eqn11}
\left\{ \begin{aligned} f ( \kappa [ u ] ) =  & \, \sigma \quad & \mbox{in} \,\, \Omega, \\
u   = & \, \epsilon \quad & \mbox{on} \,\, \Gamma,
\end{aligned} \right.
\end{equation}
where $f$ is in general setting, satisfies certain assumptions, and is normalized, and $\epsilon$ is a small positive constant. Because $\psi = \sigma \in (0, 1)$ is a constant, there is a natural subsolution $\underline{u} = \epsilon$ during the continuity process related to \eqref{eqn11}, and the estimates to \eqref{eqn11} can be independent of $\epsilon$.

When $\psi$ is not a constant, we no longer have a natural subsolution during the continuity process, and most of the methods in \cite{GS00, GSS09, GS10, GS11, GSX14} do not work in this generality.  In \cite{Sui2019, Sui-Sun}, we started a new method:
assume that there exists an asymptotic subsolution to \eqref{eqn9}, namely, there exists an admissible $\underline{u} \in C^4(\Omega) \cap C^0(\overline{\Omega})$ such that
\begin{equation} \label{eqn8}
\left\{ \begin{aligned}
f(\kappa [ \underline{u} ]) \geq  & \, \psi(x, \underline{u}) \quad & \mbox{in} \,\, \Omega, \\
\underline{u}  = & \, 0 \quad  & \mbox{on} \,\, \Gamma.
\end{aligned} \right.
\end{equation}
Then we make use of the level set of $\underline{u}$ to construct approximating Dirichlet problem:
\begin{equation} \label{eqn10}
\left\{\begin{aligned}
f ( \kappa [ u ] ) =  & \, \psi(x, u) \quad & \mbox{in} \,\, \Omega_{\epsilon}, \\
u  = & \, \epsilon \quad & \mbox{on} \,\, \Gamma_{\epsilon},
\end{aligned} \right.
\end{equation}
where
\[ \Gamma_{\epsilon} =  \{ x \in \Omega \,\big\vert\, \underline{u}(x)  = \epsilon\}, \quad
\Omega_{\epsilon} = \{ x \in \Omega \,\big\vert\, \underline{u}(x)  > \epsilon\}. \]
We also assume that $\Gamma_{\epsilon}$ is a regular boundary of $\Omega_{\epsilon}$ when $\epsilon > 0$ is sufficiently small, namely, $\Gamma_{\epsilon}$ has dimension $n - 1$, $\Gamma_{\epsilon} \in C^4$ and $|D \underline{u}| > 0$ on $\Gamma_{\epsilon}$.
Our idea is like this. The first step is to prove the existence of solutions to \eqref{eqn10}, which needs the establishment of $C^2$ estimates to \eqref{eqn10}. The feature is that we allow the dependence on $\epsilon$. Consequently, techniques for common Dirichlet problems can be applied. The next step is to derive certain uniform interior estimates up to second order which are independent of $\epsilon$ for solutions to \eqref{eqn10}. Combined with Evans-Krylov interior estimates and diagonal process, we can then obtain smooth solutions to \eqref{eqn9}.

We have finished the first step in \cite{Sui2019, Sui-Sun} under certain assumptions. In addition, we obtained uniform interior estimates up to first order for solution $u^{\epsilon}$ to \eqref{eqn10} in these papers, that is,
for any sufficiently small fixed $\epsilon_0 > 0$, we have
\begin{equation} \label{eqn15}
\Vert u^{\epsilon} \Vert_{C^1(\overline{\Omega_{\epsilon_0}})} \leq C, \quad \forall \,\, 0 < \epsilon < \frac{\epsilon_0}{2}.
\end{equation}
Here and throughout this paper, $C$ denotes a positive constant which is independent of $\epsilon$, but may depend on $\epsilon_0$.

The difficulty lies in the establishment of uniform interior estimates of second order for $u^{\epsilon}$. As known from literature, the endeavor to derive interior $C^2$ estimates fails in general context. For example, from the counterexamples of Pogorelov \cite{Po78} and Urbas \cite{Ur90}, we know that the pure interior $C^2$ estimates do not hold when $k \geq 3$. Because of this fact, we investigated prescribed scalar curvature equations in \cite{Sui2019}, using the idea of Guan-Qiu \cite{GQ17}, who established interior curvature estimate for strictly locally convex solutions to prescribed scalar curvature equations in Euclidean space. Subsequently, we studied the next best thing in \cite{Sui-Sun}: to obtain a general existence result for locally Lipschitz continuous weak admissible solutions to \eqref{eqn9}.

Despite the lack of pure interior $C^2$ estimate for general prescribed curvature equations, we seek the possibility of other types of interior $C^2$ estimates to accomplish our second step. Recently, we observe that Pogorelov type interior $C^2$ estimate will suffice to give our desired uniform interior second order estimate with the help of some suitable nested domains (see \eqref{eq2-4} and \eqref{eq2-12}), which are built by means of solution to the associated homogeneous equation
\begin{equation} \label{eqn2}
\left\{ \begin{aligned}
f ( \kappa [ v ] ) =  & \,  0 \quad & \mbox{in} \,\, \Omega, \\
v  = &  \, 0 \quad & \mbox{on} \,\, \Gamma,
\end{aligned} \right.
\end{equation}
as well as the asymptotic subsolution $\underline{u}$.
This observation is motivated by the work of Trudinger-Urbas \cite{Tru-Urb84}, who studied Pogorelov type interior second order estimate for Monge-Amp\`ere equations with general Dirichlet boundary values over uniformly convex domains. However, their estimate can not be directly applied to solution $u^{\epsilon}$ of \eqref{eqn10} with $k = n$, because it relies on the convexity of the domain and the $C^2$ norm of the boundary data (in hyperbolic space, it would also depend on the lower bound of $u$), and consequently may not bring in a uniform interior bound. Our strategy is: rather than derive Pogorelov type interior second order estimate for $u^{\epsilon}$ on $\Omega_{\epsilon}$, we consider this estimate on $\Omega^{\epsilon}_{\epsilon_0}$ (see section 2 and 3 for the definitions), which are uniformly away from $\Gamma$. As a result, we obtain the uniform interior second order estimate \eqref{eq1-15} when $k = n$, which in turn gives the first main result of this paper on prescribed Gauss curvature equations.

\begin{thm} \label{Theorem1}
For $k = n$ and $0 < \psi(x, u) \in C^{\infty} (\mathbb{H}^{n + 1})$, suppose that
there exists a strictly locally convex $\underline{u} \in C^4(\Omega) \cap C^0(\overline{\Omega})$ satisfying \eqref{eqn8} and a locally convex solution $v \in C^{2}(\Omega) \cap C^{0} (\overline{\Omega})$ satisfying
\eqref{eqn2}.
Then there exists a strictly locally convex solution $u \in C^{\infty}(\Omega) \cap C^0(\overline{\Omega})$ with $u \geq \underline{u}$ to asymptotic problem \eqref{eqn9}.
\end{thm}
We note that, in order to conduct the Pogorelov type estimate, we have assumed the existence of $v$ (c.f. \cite{Tru-Urb84}).
Such $v$ exists if $\Gamma$ bounds a uniformly convex domain.  For the special case $k = n = 2$, we do not need to assume the existence of $v$ (see Theorem 1.2 in \cite{Sui2019}) since we have obtained the pure interior $C^2$ estimate.

For general $k$, as known from Sheng-Urbas-Wang \cite{Sheng-Urbas-Wang}, Pogorelov type interior second order estimate is limited to Dirichlet problems with affine boundary values in Euclidean space. As we know, the affine boundary value spans a hyperplane, the principal curvatures of which are all zero, and hence satisfy \eqref{eqn2} in Euclidean space. The counterpart of the hyperplane in hyperbolic space is
\[ \Big\{ (x, \sqrt{R^2 - |x|^2})  \big\vert  \, |x| < R \Big\}, \]
which intersects $\partial_{\infty} \mathbb{H}^{n+1}$ at
$\Gamma = \{ x \in \mathbb{R}^n  \big\vert \, |x| = R \}$,
with $R$ a positive constant.
For this reason, we shall restrict our attention to circular asymptotic boundaries.
We shall derive the uniform interior second order estimate \eqref{eq2-33} by extending the estimate in \cite{Sheng-Urbas-Wang} to hyperbolic space. Our test function and the subsequent derivations are a little different from \cite{Sheng-Urbas-Wang} due to the hyperbolic space and the problem setting.
As a result, we obtain the following existence result over disks for general $k$.
\begin{thm}  \label{Theorem2}
Let  $k < n$ and  $\Omega = \{ x \in \mathbb{R}^n  \vert  |x| < R \}$.
Suppose that $0 < \psi(x, u) \in C^{\infty} (\mathbb{H}^{n + 1})$ satisfies
\begin{equation} \label{eqn12}
\psi_u - \frac{\psi}{u} \geq 0,
\end{equation}
there exists an admissible $\underline{u} \in C^4(\Omega) \cap C^0(\overline{\Omega})$ satisfying \eqref{eqn8} and
\begin{equation} \label{eqn13}
- \lambda (D^2 \underline{u}) \in \Gamma_{k + 1}  \quad  \mbox{near} \,\, \Gamma,
\end{equation}
and the compatibility conditions hold for sufficiently small $\epsilon > 0$.
Then there exists a unique admissible solution $u \in C^{\infty}(\Omega) \cap C^0(\overline{\Omega})$ with $u \geq \underline{u}$ to the asymptotic problem \eqref{eqn9}.
\end{thm}
We point out that condition \eqref{eqn12}, \eqref{eqn13} and the compatibility conditions (see the introduction of \cite{Sui-Sun}) are assumed to guarantee the existence of smooth admissible solution $u^{\epsilon} \geq \underline{u}$ to \eqref{eqn10} on $\overline{\Omega_{\epsilon}}$ (see Theorem 1.7 in \cite{Sui-Sun}). Moreover, \eqref{eqn12} is used to derive interior gradient estimate, which assures \eqref{eqn15}. It should also be emphasized that the assumption $\Omega = \{ x \in \mathbb{R}^n  \vert  |x| < R \}$ is assumed solely for Pogorelov type second order interior estimate. In the end, we provide an example (see also \cite{Sui-Sun}) satisfying the assumptions in our theorems.

\begin{exmp}
Let
\[\Omega = \big\{ x \in \mathbb{R}^n \big\vert \, |x| < (1 - \sigma^2)^{\frac{1}{2}}  R_1 \big\}, \]
where $\sigma \in (0, 1)$ and $R_1 > 0$ are constants. Let
$\psi = \alpha u^2$ with
\[ \alpha = \frac{\sigma_k^{\frac{1}{k}} (\sigma, \ldots, \sigma)}{(1 - \sigma)^2 R_1^2}. \]
Then we may choose $\underline{u} = \sqrt{R_1^2 - |x|^2} - \sigma R_1$.
\end{exmp}

This paper is organized as follows: we prove Theorem \ref{Theorem1} in section 2 and Theorem \ref{Theorem2} in section 3.

\medskip
\noindent
{\bf Acknowledgements} \quad
The authors would like to thank the referees for giving us many valuable comments and suggestions.

\vspace{4mm}

\section{Prescribed Gauss curvature equations}

\vspace{4mm}

In this section, we shall focus on prescribed Gauss curvature equation, that is, equation \eqref{eqn9} with $k = n$. Before we discuss this type of equation, we first present some preliminary formulae on vertical graph of $u$, which can be found in \cite{GSS09, GS10, GS11, GSX14, Sui2019, Sui-Sun}. The counterpart in Euclidean space can be found in \cite{CNSV}.

\vspace{2mm}

\subsection{Geometric quantities on vertical graph of $u$}~

\vspace{2mm}

For convenience, denote
\[ \Sigma = \{ (x, u(x)) \vert \, x \in \Omega \}. \]
The coordinate vector fields on $\Sigma$ are given by
\[  \partial_i + u_i \partial_{n + 1}, \quad i = 1, \ldots, n, \]
where $\partial_{i} = \frac{\partial}{\partial x^{i}}$ with $i = 1, \ldots, n + 1$ are the coordinate vector fields in $\mathbb{R}^{n+1}$.

When $\Sigma$ is considered as a hypersurface in $\mathbb{R}^{n + 1}$, the upward unit normal,  metric, inverse of the metric and second fundamental form are given respectively by
\[\nu = \frac{1}{w} ( - D u, 1 ), \quad w = \sqrt{ 1 + |D u |^2},  \]
\[ \tilde{g}_{ij} = \delta_{ij} + u_i u_j, \quad  \tilde{g}^{ij} =  \delta_{ij} - \frac{u_i u_j}{w^2},  \quad  \tilde{h}_{ij} = \frac{u_{ij}}{w}. \]
The Euclidean principal curvatures $\tilde{\kappa}$ are the eigenvalues of the symmetric matrix
\[ \tilde{a}_{ij} = \frac{1}{w} \gamma^{ik} u_{kl} \gamma^{lj}, \]
where
\[ \gamma^{ik} = \delta_{ik} - \frac{u_i u_k}{w ( 1 + w )}, \quad
 \gamma_{ik} = \delta_{ik} + \frac{u_i u_k}{1 + w}, \]
and
\[ \gamma^{ik} \gamma_{kj} = \delta_{ij}, \quad \gamma_{ik} \gamma_{kj} = \tilde{g}_{ij}.  \]

When $\Sigma$ is viewed as a hypersurface in $\mathbb{H}^{n + 1}$, its unit upward normal, metric, second fundamental form are given respectively by
\[ {\bf n} = u \nu, \quad  g_{ij} = \frac{1}{u^2} ( \delta_{ij} + u_i u_j ), \quad
 h_{ij} = \frac{1}{u^2 w} ( \delta_{ij} + u_i u_j + u u_{ij} ). \]
The hyperbolic principal curvatures $\kappa [u]$ are the eigenvalues of the symmetric matrix $A [u] = \{ a_{ij} \}$, where
\[ a_{ij} =  u^2 \gamma^{ik} h_{kl} \gamma^{lj}
       = \frac{1}{w} \gamma^{ik} ( \delta_{kl} + u_k u_l + u u_{kl} ) \gamma^{lj} =  \frac{1}{w} ( \delta_{ij} + u \gamma^{ik} u_{kl} \gamma^{lj} ). \]
Equation \eqref{eqn9} can be written as
\begin{equation*}
  f( \kappa [ u ] ) = f( \lambda( A[ u ] )) = F( A[ u ] ) =  \psi(x, u).
\end{equation*}

From the above formulae, we observe the following relations.
\begin{equation} \label{eq1-16}
h_{ij} = \frac{1}{u} \tilde{h}_{ij} + \frac{\nu^{n+1}}{u^2} \tilde{g}_{ij},
\end{equation}
where $\nu^{n+1} = \nu \cdot \partial_{n + 1}$ and $\cdot$ is the inner product in $\mathbb{R}^{n+1}$.
Note that formula \eqref{eq1-16} indeed holds for any local frame on any hypersurface $\Sigma$
which may not be a graph (then $u = X \cdot \partial_{n + 1}$ where $X$ is the position vector of $\Sigma$ in $\mathbb{R}^{n + 1}$). We also have the relation
\begin{equation} \label{eq1-17}
\kappa_i =  u \tilde{\kappa_i} + \nu^{n+1},  \quad i = 1, \ldots, n.
\end{equation}

\vspace{2mm}

\subsection{Pogorelov type interior second order estimate}~

\vspace{2mm}

Now we consider the following prescribed Gauss curvature equation
\begin{equation} \label{eq1-1}
\left\{ \begin{aligned}
\det  A[ u ]  = & \psi(x, u, D u) \quad & \mbox{in} \,\, \Omega, \\
u = & 0 \quad  & \mbox{on} \,\, \Gamma,
\end{aligned} \right.
\end{equation}
where $\psi$ is a smooth positive function defined on $\Omega \times (0, \infty) \times \mathbb{R}^n$.

Let $v \in C^{2}(\Omega) \cap C^{0} (\overline{\Omega})$ be a locally convex solution to the homogeneous equation
\begin{equation} \label{eq1-2}
\left\{ \begin{aligned}
\det (\delta_{kl} + v_k v_l + v v_{kl}) = &  0 \quad \mbox{in} \,\, \Omega, \\
v = & 0 \quad \mbox{on} \,\, \Gamma.
\end{aligned} \right.
\end{equation}
By the transformation $V = v^2 + |x|^2$, one can see that \eqref{eq1-2} is equivalent to
\begin{equation} \label{eq1-3}
\left\{ \begin{aligned}
\det (V_{kl}) = &  0 \quad &\mbox{in} \,\, \Omega, \\
V = & |x|^2 \quad &\mbox{on} \,\, \Gamma.
\end{aligned} \right.
\end{equation}

Assume that there exists a strictly locally convex $\underline{u} \in C^4(\Omega) \cap C^0(\overline{\Omega})$ such that
\begin{equation} \label{eq1-4}
\left\{ \begin{aligned}
\det A [ \underline{u} ] \geq  &  \psi(x, \underline{u}, D \underline{u})  \quad & \mbox{in} \,\, \Omega, \\
\underline{u}  = &  0 \quad  & \mbox{on} \,\, \Gamma.
\end{aligned} \right.
\end{equation}
Denoting $\underline{U} = {\underline{u}}^2 + |x|^2$,
we can verify that $\underline{U}$ is strictly locally convex and satisfies
\begin{equation} \label{eq1-6}
\left\{ \begin{aligned}
\det (\underline{U}_{kl}) \geq &  \Psi(x, \underline{U}, D \underline{U})   \quad &\mbox{in} \,\, \Omega, \\
\underline{U} = & |x|^2 \quad &\mbox{on} \,\, \Gamma,
\end{aligned} \right.
\end{equation}
where
\[ \Psi(x, U, D U) = 2^n \Big(\frac{|D U|^2 - 4 x \cdot D U + 4 U}{4 U - 4 |x|^2} \Big)^{\frac{n + 2}{2}} \psi\Big(x, \sqrt{U - |x|^2}, \frac{D U - 2 x}{2 \sqrt{U - |x|^2}} \Big). \]
By \eqref{eq1-3}, \eqref{eq1-6} and the maximum principle, we know that
\[ \underline{U} < V \quad \mbox{in} \,\, \Omega. \]

By Theorem 1.1 in \cite{Sui2019}, for $\epsilon > 0$ sufficiently small, there exists a strictly locally convex smooth solution $u^{\epsilon} \geq \underline{u}$ to the approximating Dirichlet problem
\begin{equation} \label{eq1-5}
\left\{\begin{aligned}
\det A [ u ]  =  & \psi(x, u, D u) \quad &\mbox{in} \,\, \Omega_{\epsilon}, \\
u  = & \epsilon \quad & \mbox{on} \,\, \Gamma_{\epsilon}.
\end{aligned} \right.
\end{equation}
Equivalently, we can verify that
$U^{\epsilon} = (u^{\epsilon})^2 + |x|^2$ is a strictly locally convex solution to the Monge-Amp\`ere equation
\begin{equation} \label{eq1-7}
\left\{ \begin{aligned}
\det (U_{kl}) = &  \Psi(x, U, D U)   \quad &\mbox{in} \,\, \Omega_{\epsilon}, \\
U = & |x|^2 + \epsilon^2 = \underline{U} \quad &\mbox{on} \,\, \Gamma_{\epsilon},
\end{aligned} \right.
\end{equation}
and $\underline{U} \leq U^{\epsilon} < V$ in $\overline{\Omega_{\epsilon}}$.

For fixed $\epsilon_0 > 0$ sufficiently small, let $r = d(\Omega_{\epsilon_0}, \Gamma_{\epsilon_0 / 2})$.  For any $x_0 \in \Omega_{\epsilon_0}$ and for $\tau > 0$, consider
\[ \phi(x) = - \tau (r^2 - |x - x_0|^2). \]
We note that on $B_{r}(x_0)$,
\[ \det (D^2 V + D^2 \phi) \leq C_1(n) \sum\limits_{k = 1}^n \tau^k M^{n - k}, \quad \mbox{where} \,\, M = \sup_{\Omega_{\epsilon_0/2}} |D^2 V|. \]
Also, we observe that on $\overline{\Omega_{\epsilon_0/2}}$, for any $0 < \epsilon < \frac{\epsilon_0}{4}$,
\[ C \geq U^{\epsilon} \geq \underline{U} \geq \frac{\epsilon_0^2}{4} + |x|^2 \quad  \mbox{and} \quad  |D U^{\epsilon}| \leq C. \]
Thus, for any $0 < \epsilon < \frac{\epsilon_0}{4}$, we have
\[ \inf_{\Omega_{\epsilon_0/2}} \Psi(x, U^{\epsilon}, D U^{\epsilon}) \geq C^{-1} > 0. \]

Now we can choose $\tau$ sufficiently small (depending on $\epsilon_0$) such that for any $x_0 \in \Omega_{\epsilon_0}$, we have on $B_{r}(x_0)$,
\[ \det (D^2 V + D^2 \phi) \leq \det D^2 U^{\epsilon}, \quad \forall \,\, 0 < \epsilon < \frac{\epsilon_0}{4}. \]
Since $V + \phi = V > U^{\epsilon}$ on $\partial B_r(x_0)$, by the maximum principle, $V + \phi \geq U^{\epsilon}$ over $B_r(x_0)$. In particular, we have
\[ V (x_0) - \tau r^2 = (V + \phi) (x_0) \geq U^{\epsilon} (x_0). \]
Thus,
\begin{equation} \label{eq1-14}
\inf_{\Omega_{\epsilon_0}} (V - U^{\epsilon}) \geq \tau r^2, \quad \forall \,\, 0 < \epsilon < \frac{\epsilon_0}{4}.
\end{equation}

Now, set
\[ \Omega^{\epsilon}_{\epsilon_0} = \{ x\in \Omega_{\epsilon} \,\big\vert\, (V - U^{\epsilon})(x) > \frac{1}{2} \tau r^2 \} \quad  \mbox{for} \,\, 0 < \epsilon < \frac{\epsilon_0}{4}, \]
and
\[  \hat{\Omega}_{\epsilon_0} = \{ x \in \Omega \,\big\vert\, (V - \underline{U})(x) \geq \frac{1}{2} \tau r^2 \}.  \]
It is easy to verify that
\begin{equation} \label{eq2-4}
\Omega_{\epsilon_0} \subset \subset \Omega^{\epsilon}_{\epsilon_0} \subset \hat{\Omega}_{\epsilon_0} \subset \subset \Omega, \quad \forall \,\, 0 < \epsilon < \frac{\epsilon_0}{4}.
\end{equation}
Denote
\[ \delta_{\epsilon_0} = \min\limits_{\hat{\Omega}_{\epsilon_0}} \underline{u} > 0. \]
Then, for any $0 < \epsilon < \min\{\epsilon_0 / 4, \delta_{\epsilon_0} / 2\}$,  we have
\[ |x|^2 + (\delta_{\epsilon_0})^2 \leq U^{\epsilon} \leq C \quad \mbox{and} \quad  |D U^{\epsilon}| \leq C  \quad \mbox{on} \,\, \Omega^{\epsilon}_{\epsilon_0}, \]
and
\[ V - U^{\epsilon} = \frac{1}{2} \tau r^2 \quad \mbox{on} \,\, \partial\Omega^{\epsilon}_{\epsilon_0}. \]

Now, we apply Trudinger-Urbas' method \cite{Tru-Urb84} to obtain second order Pogorelov type interior estimate.   Consider in $\Omega_{\epsilon_0}^{\epsilon}$,
\[ \ln \eta + \frac{\beta}{2} |D U^{\epsilon}|^2 + \ln U^{\epsilon}_{\xi \xi}, \quad \mbox{where} \,\, \eta = V - U^{\epsilon} - \frac{1}{2} \tau r^2. \]
The maximum value must be achieved at an interior point in $\Omega_{\epsilon_0}^{\epsilon}$. We assume it to be $0$ and in direction $\xi = {\partial}_1$. In addition, we assume that $\big(U^{\epsilon}_{ij}(0)\big)$ is diagonal after we rotate $\partial_2, \ldots, \partial_n$. For convenience, we omit the superscript $\epsilon$ in $U^{\epsilon}$. Thus, the function
\[  \ln \eta + \frac{\beta}{2} |D U|^2 + \ln U_{11}   \]
achieves its maximum at $0$, at which, we have
\begin{equation} \label{eq1-8}
\frac{\eta_i}{\eta} + \beta U_i U_{ii} + \frac{U_{11i}}{U_{11}} = 0,
\end{equation}
\begin{equation} \label{eq1-9}
\frac{U^{ij} \eta_{ij}}{\eta} - \frac{U^{ij} \eta_i \eta_j}{\eta^2} + \beta U_k U^{ij} U_{kij} + \beta \Delta U + \frac{U^{ij} U_{11ij}}{U_{11}} - \frac{U^{ij} U_{11i} U_{11j}}{U_{11}^2} \leq 0,
\end{equation}
where $U^{ij} U_{jk} = \delta^i_k$. Since $U$ is strictly locally convex, it suffices to give an upper bound for $\eta U_{11}$. We may assume that $U_{11} \geq 1$.
Write equation \eqref{eq1-7} as
\[ \ln \det (D^2 U) = \ln \Psi(x, U, D U) = \tilde{\Psi}(x, U, D U). \]
Differentiate this equation, we obtain
\begin{equation} \label{eq1-10}
U^{ij} U_{ijk} = \tilde{\Psi}_{x_k} + \tilde{\Psi}_U U_k + \tilde{\Psi}_{U_i} U_{ik},
\end{equation}
\begin{equation} \label{eq1-11}
U^{ij} U_{ij11} - U^{ik} U^{jl} U_{ij1} U_{kl1} \geq \tilde{\Psi}_{U_i} U_{i11} - C U_{11}^2.
\end{equation}
Also, we have
\begin{equation} \label{eq1-12}
U^{ij} \eta_{ij} = U^{ij} (V_{ij} - U_{ij}) \geq - n.
\end{equation}
Taking \eqref{eq1-8} and \eqref{eq1-10}--\eqref{eq1-12} into \eqref{eq1-9} yields,
\begin{equation} \label{eq1-13}
- \frac{U^{ii} \eta_i^2}{\eta^2} - C \beta + \beta U_{11} + \frac{U^{ii} U^{jj} U_{ij1}^2}{U_{11}} - \frac{C}{\eta} - C U_{11} - \frac{U^{ii} U_{11i}^2}{U_{11}^2} \leq 0.
\end{equation}
Assuming $\eta U_{11} \geq 1$, again by \eqref{eq1-8}, we have
\[ \begin{aligned}
\frac{U^{ii} \eta_i^2}{\eta^2} = & \frac{U^{11} \eta_1^2}{\eta^2} + \sum_{i \neq 1} U^{ii} \Big( \beta U_i U_{ii} + \frac{U_{11i}}{U_{11}} \Big)^2 \\
= &  \frac{U^{11} \eta_1^2}{\eta^2} + \sum_{i \neq 1}  \Big( \beta^2 U_i^2 U_{ii} + 2 \beta U_i \big( - \frac{\eta_i}{\eta} - \beta U_i U_{ii} \big) + U^{ii} \frac{U_{11i}^2}{U_{11}^2} \Big) \\
\leq & \frac{C}{\eta} + \frac{C \beta}{\eta} + \sum_{i \neq 1} \frac{U^{ii} U_{11i}^2}{U_{11}^2}.
\end{aligned} \]
Hence, \eqref{eq1-13} reduces to
\begin{equation} \label{eq1-18}
\frac{U^{ii} U^{jj} U_{ij1}^2}{U_{11}} - \frac{U^{ii} U_{11i}^2}{U_{11}^2} - \sum\limits_{i \neq 1} \frac{U^{ii} U_{11i}^2}{U_{11}^2} + \beta U_{11} - C U_{11} - \frac{C}{\eta} - \frac{C \beta}{\eta} \leq 0.
\end{equation}
Also,
\[ \begin{aligned}
\frac{U^{ii} U^{jj} U_{ij1}^2}{U_{11}} = & \frac{U^{ii} U^{11} U_{i11}^2}{U_{11}} + \sum\limits_{j \neq 1} \frac{U^{11} U^{jj} U_{1j1}^2}{U_{11}} + \sum\limits_{j \neq 1, i \neq 1} \frac{U^{ii} U^{jj} U_{ij1}^2}{U_{11}} \\
\geq &  \frac{U^{ii} U_{11i}^2}{U_{11}^2} + \sum\limits_{j \neq 1} \frac{U^{jj} U_{11j}^2}{U_{11}^2}.
\end{aligned} \]
Hence \eqref{eq1-18} becomes
\[ \beta U_{11} - C U_{11} - \frac{C}{\eta} - \frac{C \beta}{\eta} \leq 0. \]
Choosing $\beta$ sufficiently large, we thus obtain an upper bound $\eta U_{11} \leq C$.
Therefore, by \eqref{eq1-14}, we arrive at
\begin{equation} \label{eq1-15}
\sup_{\Omega_{\epsilon_0}} |D^2 U^{\epsilon}| \leq C, \quad \forall \,\, 0 < \epsilon < \min\{\epsilon_0 / 4, \delta_{\epsilon_0} / 2\}.
\end{equation}

\vspace{4mm}

\section{Prescribed $k$th Weingarten curvature equations}

\vspace{4mm}

In this section, we consider prescribed $k$th Weingarten curvature equation
\begin{equation} \label{eq2-1}
\left\{ \begin{aligned}
f ( \kappa [ u ] ) =  & F( A [ u ] ) = \sigma_k^{1/k} ( A [ u ] ) = \psi(x, u) \quad & \mbox{in} \,\, \Omega, \\
u  = &   0 \quad & \mbox{on} \,\, \Gamma.
\end{aligned} \right.
\end{equation}

Motivated by the idea of Trudinger-Urbas \cite{Tru-Urb84}, we assume that there exists a solution $v \in C^{2}(\Omega) \cap C^{0} (\overline{\Omega})$ with $\kappa[v] \in \overline{\Gamma_k}$ to the homogeneous equation
\begin{equation} \label{eq2-2}
\left\{ \begin{aligned}
f (\kappa [ v ] ) = &  0 \quad \mbox{in} \,\, \Omega, \\
v = & 0 \quad \mbox{on} \,\, \Gamma.
\end{aligned} \right.
\end{equation}
Since $\kappa[\underline{u}] \in \Gamma_k$ and $\kappa[v] \in \partial \Gamma_k$, we have
\[ \underline{u} < v \quad \mbox{in} \,\, \Omega. \]
In fact, suppose $\underline{u} - v$ achieves a nonnegative maximum at $x_0 \in \Omega$. Then we have
\[ \underline{u}(x_0) \geq v(x_0),\quad D \underline{u}(x_0) = D v(x_0), \quad D^2\underline{u}(x_0) \leq D^2 v(x_0). \]
But then at $x_0$,
\[ \delta_{ij} + v  {\gamma}^{ik}  v_{kl}  \gamma^{lj}
  \geq  \delta_{ij} + v {\gamma}^{ik}  \underline{u}_{kl} \gamma^{lj}
   =  \Big( 1 - \frac{v}{\underline{u}}\Big) \delta_{ij} + \frac{v}{\underline{u}} \Big( \delta_{ij} + \underline{u}  \gamma^{ik} \underline{u}_{kl} \gamma^{lj} \Big), \]
which leads to a contradiction.

For $\epsilon > 0$ sufficiently small, let $u^{\epsilon} \geq \underline{u}$ be a smooth admissible solution to the approximating Dirichlet problem \eqref{eqn10}.
Again, we have $\underline{u} \leq u^{\epsilon} < v$ in $\overline{\Omega_{\epsilon}}$.

For fixed $\epsilon_0 > 0$ sufficiently small, let $r = d(\Omega_{\epsilon_0}, \Gamma_{\epsilon_0 / 2})$.  For any $x_0 \in \Omega_{\epsilon_0}$ and for sufficiently small $\tau > 0$ (depending on $\epsilon_0$), consider on $B_{r}(x_0)$,
\[ v^{\tau} = \sqrt{ v^2 - \tau (r^2 - |x - x_0|^2)} > 0. \]
Since
\[  D v^{\tau} = D v + \mathcal{O}(\tau), \quad |D v^{\tau}|^2 = |D v|^2 + \mathcal{O}(\tau), \]
we have
\[ \delta_{il} - \frac{v^{\tau}_i v^{\tau}_l}{1 + |D v^{\tau}|^2} = \delta_{il} - \frac{v_i v_l}{1 + |D v|^2} + \mathcal{O}(\tau). \]
Also,
\[  \delta_{ij} + v^{\tau}_i v^{\tau}_j  + v^{\tau} v^{\tau}_{ij} = \frac{1}{2} \Big( v^2 + |x|^2 - \tau (r^2 - |x - x_0|^2) \Big)_{ij} = \delta_{ij} + v_i v_j  + v v_{ij} + \tau \delta_{ij}.  \]
Therefore, we have
\begin{equation} \label{eq2-3}
\begin{aligned}
& \sigma_k \Big( \big( \delta_{il} - \frac{v^{\tau}_i v^{\tau}_l}{1 + |D v^{\tau}|^2} \big) \big( \delta_{lj} + v^{\tau}_l v^{\tau}_j + v^{\tau} v^{\tau}_{lj} \big) \Big) \\
= &  \sigma_k \Big( \big( \delta_{il} - \frac{v_i v_l}{1 + |D v|^2} \big) \big( \delta_{lj} + v_l v_j + v v_{lj} \big) + \mathcal{O}(\tau) \Big)   \\
\leq & \sigma_k \Big( \big( \delta_{il} - \frac{v_i v_l}{1 + |D v|^2} \big) \big( \delta_{lj} + v_l v_j + v v_{lj} \big) \Big) + C \tau = C \tau.
\end{aligned}
\end{equation}
Here $C$ in the last row of \eqref{eq2-3} depends on $n$, $k$, $\epsilon_0$ and $\Vert v \Vert_{C^2(\Omega_{\epsilon_0 / 2})}$. The argument is similar as in section 2.

In \cite{Sui-Sun}, we have obtained the $C^0$ estimate
\[ \epsilon \leq \underline{u} \leq u^{\epsilon} \leq C_0 \quad \text{on  } \overline{\Omega_{\epsilon}}, \]
where $C_0$ is independent of $\epsilon$. Therefore we have
\[ \inf_{\Omega_{\epsilon_0/2}} \psi(x, u^{\epsilon}) \geq C^{-1} > 0, \quad \forall \,\, 0 < \epsilon < \frac{\epsilon_0}{2}. \]
By \eqref{eq2-3}, we can choose $\tau$ sufficiently small (depending on $\epsilon_0$) such that for any $x_0 \in \Omega_{\epsilon_0}$, we have on $B_{r}(x_0)$,
\begin{equation} \label{eq2-6}
\sigma_k (A[ v^{\tau} ]) < \frac{(v^{\tau})^k}{(u^{\epsilon})^k} \sigma_k (A[ u^{\epsilon} ]), \quad \forall \,\, 0 < \epsilon < \frac{\epsilon_0}{2}.
\end{equation}
Since $v^{\tau} = v > u^{\epsilon}$ on $\partial B_r(x_0)$, we must have $v^{\tau} > u^{\epsilon}$ over $B_r(x_0)$. In fact, suppose
$v^{\tau} - u^{\epsilon}$ achieves a nonpositive minimum at $x_1 \in B_r(x_0)$. Then we have
\[ v^{\tau}(x_1) \leq u^{\epsilon}(x_1),\quad D v^{\tau}(x_1) = D u^{\epsilon}(x_1), \quad D^2 v^{\tau}(x_1) \geq D^2 u^{\epsilon}(x_1). \]
But then at $x_1$,
\[\begin{aligned}
 & \sigma_k \Big( \frac{A[ v^{\tau} ]}{v^{\tau}} \Big) =   \sigma_k \Big( \frac{1}{w} \big( \frac{\delta_{ij}}{v^{\tau}} +  \gamma^{il} v^{\tau}_{lm} \gamma^{mj} \big) \Big)
\\ \geq  \, &  \sigma_k \Big( \frac{1}{w} \big( \frac{\delta_{ij}}{u^{\epsilon}} +  \gamma^{il} u^{\epsilon}_{lm} \gamma^{mj} \big) \Big)
=  \sigma_k \Big( \frac{A[ u^{\epsilon} ]}{u^{\epsilon}} \Big),
\end{aligned} \]
contradicting \eqref{eq2-6}.
In particular, we have
\[ v^2 (x_0) - \tau r^2  > (u^{\epsilon})^2 (x_0), \]
which implies that
\begin{equation} \label{eq2-14}
\inf_{\Omega_{\epsilon_0}} \big( v^2 - (u^{\epsilon})^2 \big) \geq 2 c := \tau r^2 > 0, \quad \forall \,\, 0 < \epsilon < \frac{\epsilon_0}{2}.
\end{equation}
Note that $c$ depends on $\epsilon_0$ but is independent of $\epsilon$.

Similar to last section, define for $0 < \epsilon < \frac{\epsilon_0}{2}$,
\[ \Omega^{\epsilon}_{\epsilon_0} = \{ x\in \Omega_{\epsilon} \,\big\vert\,  \big( v^2 - (u^{\epsilon})^2 \big) (x) > c \}   \quad \mbox{and} \quad
  \hat{\Omega}_{\epsilon_0} = \{ x \in \Omega \,\big\vert\, (v^2 - \underline{u}^2)(x) \geq c \}.  \]
We can verify that
\begin{equation} \label{eq2-12}
\Omega_{\epsilon_0} \subset \subset \Omega^{\epsilon}_{\epsilon_0} \subset \hat{\Omega}_{\epsilon_0} \subset \subset \Omega, \quad \forall \,\, 0 < \epsilon < \frac{\epsilon_0}{2}.
\end{equation}
Denote
\[ \delta_{\epsilon_0} = \min\limits_{\hat{\Omega}_{\epsilon_0}} \underline{u} > 0. \]
We know that for any $0 < \epsilon < \min\{\epsilon_0 / 2, \delta_{\epsilon_0} / 2\} = \delta_{\epsilon_0} / 2$,
\[ \delta_{\epsilon_0} \leq u^{\epsilon} \leq C_0, \quad  |D u^{\epsilon}| \leq C  \quad \mbox{on} \,\, \Omega^{\epsilon}_{\epsilon_0}  \]
and
\[ v^2 - (u^{\epsilon})^2 = c \quad \mbox{on} \,\, \partial\Omega^{\epsilon}_{\epsilon_0}. \]

Now we want to establish Pogorelov type interior curvature estimate for the graph of $u^{\epsilon}$ on $\Omega_{\epsilon_0}^{\epsilon}$ for any $0 < \epsilon < \delta_{\epsilon_0} / 2$, i.e.
\begin{equation} \label{eq2-7}
\max_{1 \leq i \leq n} \big| \kappa_i [ u^{\epsilon}] (x) \big| \,\leq \, \frac{C}{\big( v^2 - (u^{\epsilon})^2 - c \big)^b}, \quad \forall \,\, x \in \Omega_{\epsilon_0}^{\epsilon}, \quad \forall \,\, 0 < \epsilon < \frac{\delta_{\epsilon_0}}{2},
\end{equation}
where $b \geq 1$ is a constant to be determined later.
For convenience, we omit the superscript in $u^{\epsilon}$ and denote the graph of $u$ over $\Omega_{\epsilon_0}^{\epsilon}$ by $\Sigma$.

At this point, we assume that $v = \sqrt{R^2 - |x|^2}$. Consequently, $\Omega$ has to be
\[ \Omega = \{ x \in \mathbb{R}^n  \vert \, |x| < R \}. \]
Note that $\kappa[v] = (0, \ldots, 0)$, which can be computed easily by formula \eqref{eq1-17}. Hence $v$ satisfies \eqref{eq2-2}.

In order to perform calculations on $\Sigma$, we may need the following preliminary formulae (see \cite{GS11, Sui2019, Sui-Sun}).

\vspace{4mm}

\subsection{Calculations on $\Sigma$}~

\vspace{2mm}

For a hypersurface $\Sigma$ in $\mathbb{H}^{n+1}$, let $g$ and $\nabla$ be the induced metric and Levi-Civita connection on $\Sigma$. Let $\tilde{g}$ and $\tilde{\nabla}$ be the metric and Levi-Civita connection on $\Sigma$ induced from $\mathbb{R}^{n+1}$. Let $X$ be the position vector of $\Sigma$ in $\mathbb{R}^{n + 1}$ and set $u = X \cdot \partial_{n + 1}$. Since the associated Christoffel symbols have the following relation
\[ \Gamma_{ij}^k = \tilde{\Gamma}_{ij}^k - \frac{1}{u} (u_i \delta_{kj} + u_j \delta_{ik} - \tilde{g}^{kl} u_l \tilde{g}_{ij}), \]
we can prove that for any $w \in C^2(\Sigma)$ and any local frame on $\Sigma$,
\begin{equation} \label{eq2-10}
\nabla_{ij} w = (w_i)_j - \Gamma_{ij}^k w_k = \tilde{\nabla}_{ij} w + \frac{1}{u}( u_i w_j + u_j w_i - \tilde{g}^{kl} u_l w_k \tilde{g}_{ij}).
\end{equation}

\begin{lemma}  \label{Lemma2-2}
When $\Sigma$ is viewed as a hypersurface in $\mathbb{R}^{n+1}$, we have
\[ u_i = \tau_i \cdot \partial_{n + 1}, \quad (x^{\alpha})_i = \tau_i \cdot \partial_{\alpha},  \quad \alpha = 1, \ldots, n, \]

\[   \tilde{g}^{kl} u_k u_l  = |\tilde\nabla u|^2 = 1 - (\nu^{n+1})^2, \]

\[  \tilde{\nabla}_{ij} u = \tilde{h}_{ij} \nu^{n+1}, \quad \tilde{\nabla}_{ij} x^{\alpha} = \tilde{h}_{ij} \nu^{\alpha}, \quad \alpha = 1, \ldots, n, \]

\[  (\nu^{n+1})_i = - \tilde{h}_{ij} \,\tilde{g}^{j k} u_k,  \]
and
\[   \tilde{\nabla}_{ij} \nu^{n+1} = - \tilde{g}^{kl} ( \nu^{n+1} \tilde{h}_{il} \tilde{h}_{kj} + u_l \tilde{\nabla}_k \tilde{h}_{ij} ),  \]
where $\tau_1, \ldots, \tau_n$ is any local frame on $\Sigma$, $X = (x^1, \ldots, x^n, u)$ is the position vector field on $\Sigma$ and $\nu = (\nu^1, \ldots, \nu^{n + 1})$ is the unit normal vector field to $\Sigma$ in $\mathbb{R}^{n+1}$.
\end{lemma}
\begin{proof}
The proof can be found in \cite{Sui2019}.
\end{proof}

\vspace{1mm}

For convenience, denote
\[ f_i = \frac{\partial f}{\partial \kappa_i} \quad \text{and} \quad F^{ij} = \frac{\partial F}{\partial a_{ij}}. \]

\begin{lemma}  \label{Lemma2-1}
Let $\Sigma$ be an admissible hypersurface in $\mathbb{H}^{n+1}$ satisfying
$f ( \kappa ) = \psi$.
Then in a local orthonormal frame on $\Sigma$, we have
\[
\begin{aligned}
F^{ij} \nabla_{ij} \nu^{n+1}
= & - \nu^{n+1} F^{ij} h_{ik} h_{kj} + \big( 1 + (\nu^{n+1})^2 \big) F^{ij} h_{ij} - \nu^{n+1} \sum f_i \\ & - \frac{2}{u^2} F^{ij} h_{jk} u_i u_k + \frac{2 \nu^{n+1}}{u^2} F^{ij} u_i u_j - \frac{u_k}{u} \psi_k.
\end{aligned}
\]
\end{lemma}
\begin{proof}
The proof can be found in \cite{Sui2019}.
\end{proof}

\vspace{4mm}

\subsection{Pogorelov type interior curvature estimate}~

\vspace{2mm}

Note that on $\Omega_{\epsilon_0}^{\epsilon}$, in view of \eqref{eq2-12}, we have
\[ \nu^{n+1} = \frac{1}{\sqrt{1 + |D u|^2}} \geq 2 a > 0, \]
where $a$ is a positive constant depending on $\epsilon_0$ but independent of $\epsilon$.
Let $\kappa_{\max} (X)$ be the largest principal curvature of $\Sigma$ at $X$. Consider the test function
\begin{equation*}
M_0 = \sup\limits_{X \in\Sigma}
(v^2 - u^2 - c)^b \frac{\kappa_{\max}(X)}{{\nu}^{n+1} - a} e^{\frac{\beta}{u}},
\end{equation*}
where $b \geq 1$ and $\beta > 0$ are constants to be determined. Obviously, $M_0 > 0$ is attained at an interior point $X_0 \in \Sigma$.

Let $\tau_1, \ldots, \tau_n$ be a local orthonormal frame with respect to hyperbolic metric about
$X_0$ such that $h_{ij}(X_0) = \kappa_i \delta_{ij}$, where
$\kappa_1 \geq \ldots \geq \kappa_n$ are the principal curvatures of
$\Sigma$ at $X_0$. Then
\[ \ln h_{11} - \ln ( {\nu}^{n+1} - a ) + \frac{\beta}{u} + b \ln (v^2 - u^2 - c) \]
also achieves a
local maximum at $X_0$, at which, we have
\begin{equation} \label{eq2-16}
\frac{h_{11i}}{h_{11}} - \frac{\nabla_i \nu^{n + 1}}{\nu^{ n + 1 } - a} - \beta \frac{u_i}{u^2} + b \frac{(v^2 - u^2)_i}{v^2 - u^2 - c} = 0,
\end{equation}
\begin{equation} \label{eq2-17}
\begin{aligned}
& \frac{F^{ii} h_{11ii}}{h_{11}} - \frac{F^{ii} h_{11i}^2}{h_{11}^2} - \frac{F^{ii} \nabla_{ii} \nu^{n + 1}}{\nu^{n + 1} - a} + \frac{F^{ii} (\nu^{n + 1})_i^2}{(\nu^{n + 1} - a)^2} \\
&- \beta \frac{F^{ii} \nabla_{ii} u}{u^2} + \beta F^{ii} \frac{2 u_i^2}{u^3} + b \frac{F^{ii} \nabla_{ii} (v^2 - u^2)}{v^2 - u^2 - c} - b F^{ii} \frac{( v^2 - u^2)_i^2}{( v^2 - u^2 - c )^2} \leq 0.
\end{aligned}
\end{equation}

Note that $v^2 = R^2 - |x|^2$. By \eqref{eq2-10}, Lemma \ref{Lemma2-2} and Lemma \ref{Lemma2-1}, we have
\begin{equation} \label{eq2-24}
F^{ii} \nabla_{ii} u =  u \nu^{n+1} \psi + \frac{2}{u} \sum f_i u_i^2 - u \sum f_i,
\end{equation}
\begin{equation} \label{eq2-11}
\begin{aligned}
& F^{ii} \nabla_{ii} (v^2 - u^2) =  F^{ii} \nabla_{ii} (R^2 - |X|^2) \\
= & F^{ii} \Big( - \tilde{\nabla}_{ii} |X|^2  - \frac{2 u_i}{u} (|X|^2)_i + \sum_{k = 1}^n \frac{u_k}{u} (|X|^2)_k  \Big) \\
= & F^{ii} \Big( - (2 X \cdot \tau_i)_i + \tilde{\Gamma}_{ii}^k 2 X \cdot \tau_k   - \frac{2 u_i}{u} (2 X \cdot \tau_i) + \sum_{k = 1}^n \frac{u_k}{u} (2 X \cdot \tau_k) \Big) \\
=  & F^{ii} \Big( - 2 u^2 - 2 X \cdot \nu \tilde{h}_{ii} - \frac{2 u_i}{u} (2 X \cdot \tau_i) + \sum_{k = 1}^n \frac{u_k}{u} (2 X \cdot \tau_k) \Big) \\
=  & F^{ii} \Big( - 2 u^2 - 2 X \cdot \nu (\kappa_i - \nu^{n + 1}) u - \frac{2 u_i}{u} (2 X \cdot \tau_i) + \sum_{k = 1}^n \frac{u_k}{u} (2 X \cdot \tau_k) \Big) \\
= &  F^{ii} \bigg( - 2 u^2 + 2 u (X \cdot \nu) (\nu \cdot \partial_{n + 1})- \frac{4 u_i}{u} (X \cdot \tau_i) \\
&  + 2 u \sum_{k = 1}^n \Big( \frac{\tau_k}{u} \cdot \partial_{n + 1} \Big) \Big( X \cdot \frac{\tau_k}{u} \Big) \bigg) - 2 (X \cdot \nu) u \sum f_i \kappa_i \\
=  &  - 2 (X \cdot \nu)  u \psi - 4 f_i \frac{u_i}{u} (X \cdot \tau_i),
\end{aligned}
\end{equation}
and
\begin{equation} \label{eq2-13}
\begin{aligned}
F^{ii} \nabla_{ii} \nu^{n+1}
= & - \nu^{n+1} \sum f_i \kappa_i^2 + \big( 1 + (\nu^{n+1})^2 \big) \psi - \nu^{n+1} \sum f_i \\ & -  \sum \frac{2 f_i \kappa_i u_i^2}{u^2} + \sum \frac{2 \nu^{n+1} f_i u_i^2}{u^2}  - \frac{u_k}{u} \psi_k.
\end{aligned}
\end{equation}
Also, by Gauss equation, we have the following commutation formula,
\begin{equation} \label{eq2-8}
h_{ii11} = h_{11ii} + ( \kappa_i \kappa_1 - 1 )( \kappa_i - \kappa_1 ).
\end{equation}
Differentiating equation \eqref{eq2-1} twice, we obtain
\begin{equation}  \label{eq2-9}
F^{ii} h_{ii11} + F^{ij, rs} h_{ij1} h_{rs1} = \psi_{11}  \geq  - C \kappa_1.
\end{equation}
Taking \eqref{eq2-24}--\eqref{eq2-9} into \eqref{eq2-17} yields,
\begin{equation} \label{eq2-18}
\begin{aligned}
& \Big( \kappa_1 - \frac{1 + (\nu^{n + 1})^2}{\nu^{n + 1} - a} - \frac{\beta \nu^{n + 1}}{u} - \frac{2 b (X \cdot \nu) u}{v^2 - u^2 - c} \Big) \psi - C \\
 & + \Big( \frac{\beta}{u} + \frac{a}{\nu^{n+1} - a} \Big) \sum f_i
 + \frac{a}{\nu^{n+1} - a} \sum f_i \kappa_i^2  \\
 & + \frac{2}{\nu^{n+1} - a}  \sum f_i \kappa_i \frac{u_i^2}{u^2}
- \frac{2 \nu^{n+1}}{\nu^{n+1} - a} \sum f_i \frac{u_i^2}{u^2} \\
& - b \sum f_i \frac{(v^2 - u ^2)_i^2}{(v^2 - u^2 - c)^2} - \frac{4 b f_i u_i (X \cdot \tau_i)}{u (v^2 - u^2 - c)} \\
& - \frac{F^{ij, rs} h_{ij1} h_{rs1}}{\kappa_1} - \frac{F^{ii} h_{11i}^2}{\kappa_1^2} + \frac{F^{ii} (\nu^{n+1})_i^2}{(\nu^{n+1} - a)^2} \leq 0.
\end{aligned}
\end{equation}
By Lemma \ref{Lemma2-2} and Cauchy-Schwartz inequality, we have
\begin{equation*}
\begin{aligned}
 & \frac{2}{\nu^{n+1} - a}  \sum f_i \kappa_i \frac{u_i^2}{u^2} -  \frac{2 \nu^{n+1}}{\nu^{n+1} - a}  \sum f_i \frac{u_i^2}{u^2} \\
\geq & - \frac{2}{\nu^{n+1} - a}  \sum f_i |\kappa_i| - \frac{2}{\nu^{n+1} - a}  \sum f_i \\
\geq & -  \frac{4}{ a (\nu^{n+1} - a)}  \sum f_i  - \frac{a}{4 (\nu^{n+1} - a)} \sum f_i \kappa_i^2 - \frac{2}{\nu^{n+1} - a} \sum f_i.
\end{aligned}
\end{equation*}
Hence, \eqref{eq2-18} reduces to
\begin{equation} \label{eq2-18-1}
\begin{aligned}
& \Big( \kappa_1 - \frac{1 + (\nu^{n + 1})^2}{\nu^{n + 1} - a} - \frac{\beta \nu^{n + 1}}{u} - \frac{2 b (X \cdot \nu) u}{v^2 - u^2 - c} \Big) \psi - C \\
 & + \Big( \frac{\beta}{u} + \frac{a - 2 - 4 a^{- 1}}{\nu^{n+1} - a} \Big) \sum f_i
 + \frac{3 a}{4(\nu^{n+1} - a)} \sum f_i \kappa_i^2  \\
& - b \sum f_i \frac{(v^2 - u ^2)_i^2}{(v^2 - u^2 - c)^2} - \frac{4 b f_i u_i (X \cdot \tau_i)}{u (v^2 - u^2 - c)} \\
& - \frac{F^{ij, rs} h_{ij1} h_{rs1}}{\kappa_1} - \frac{F^{ii} h_{11i}^2}{\kappa_1^2} + \frac{F^{ii} (\nu^{n+1})_i^2}{(\nu^{n+1} - a)^2} \leq 0.
\end{aligned}
\end{equation}

Let $\theta \in (0, 1)$ be a constant which will be specified later. We divide our discussion into two cases.

{\bf Case (i)}. Assume that $\kappa_n \leq - \theta \kappa_1$.
By \eqref{eq2-16} and Cauchy-Schwartz inequality, we have
\begin{equation} \label{eq2-20}
\begin{aligned}
& \frac{F^{ii} h_{11i}^2}{\kappa_1^2} \leq (1 + \delta_1) f_i \frac{(\nu^{n + 1})_i^2}{(\nu^{n + 1} - a)^2} \\
 & + 2 (1 + \delta_1^{- 1}) \beta^2 f_i \frac{u_i^2}{u^4} + 2 (1 + \delta_1^{- 1}) b^2 f_i \frac{(v^2 - u^2)_i^2}{(v^2 - u^2 - c)^2},
\end{aligned}
\end{equation}
where $\delta_1$ is a positive constant to be determined later.
Also, we have
\begin{equation} \label{eq2-21}
- \frac{4 b f_i u_i (X \cdot \tau_i)}{u (v^2 - u^2 - c)} = \frac{2 b f_i u_i (v^2 - u^2)_i}{u (v^2 - u^2 - c)}
\geq  - \sum f_i - b^2 f_i \frac{ (v^2 - u^2)_i^2}{(v^2 - u^2 - c)^2}.
\end{equation}
Taking \eqref{eq2-20} and \eqref{eq2-21} into \eqref{eq2-18-1} and noting that
$- F^{ij, rs} h_{ij1} h_{rs1} \geq 0$,  we obtain
\begin{equation} \label{eq2-22}
\begin{aligned}
& \Big( \kappa_1  - \frac{1 + (\nu^{n + 1})^2}{\nu^{n + 1} - a} - \frac{\beta \nu^{n + 1}}{u} - \frac{2 b (X \cdot \nu) u}{v^2 - u^2 - c} \Big) \psi - C \\
& + \Big( \frac{\beta}{u} + \frac{a - 2 - 4 a^{- 1}}{\nu^{n + 1} - a} - 1
- 2 (1 + \delta_1^{- 1}) \frac{\beta^2}{u^2} - \frac{ C b + C (3 + 2 \delta_1^{- 1}) b^2 }{(v^2 - u^2 - c)^2} \Big) \sum f_i \\
& + \frac{3 a}{4 (\nu^{n + 1} - a)}  \sum f_i \kappa_i^2
 - \delta_1 \frac{F^{ii} (\nu^{n+1})_i^2}{(\nu^{n+1} - a)^2} \leq 0.
\end{aligned}
\end{equation}
By Lemma \ref{Lemma2-2} and \eqref{eq1-16}, we know that
\[ (\nu^{n + 1})_i = \frac{u_i}{u} (\nu^{n + 1} - \kappa_i), \]
and consequently,
\begin{equation} \label{eq2-31}
\frac{F^{ii} (\nu^{n+1})_i^2}{(\nu^{n+1} - a)^2} \leq \frac{2}{(\nu^{n + 1} - a)^2} \sum f_i \kappa_i^2 + \frac{2}{(\nu^{n + 1} - a)^2} \sum f_i.
\end{equation}
Taking \eqref{eq2-31} into \eqref{eq2-22} and choosing $\delta_1 = \frac{a^2}{8}$, we have
\begin{equation*}
\begin{aligned}
& \Big( \kappa_1  - \frac{1 + (\nu^{n + 1})^2}{\nu^{n + 1} - a} - \frac{\beta \nu^{n + 1}}{u} - \frac{2 b (X \cdot \nu) u}{v^2 - u^2 - c} \Big) \psi - C  +  \frac{a}{2 (\nu^{n + 1} - a)}  \sum f_i \kappa_i^2  \\
& + \Big( \frac{\beta}{u} + \frac{\frac{3}{4} a - 2 - \frac{4}{a}}{\nu^{n + 1} - a} - 1 - \frac{2 (1 + 8 a^{- 2}) \beta^2}{u^2} - C \frac{ b + (3 + 16 a^{- 2}) b^2 }{(v^2 - u^2 - c)^2} \Big) \sum f_i
 \leq 0.
\end{aligned}
\end{equation*}
Since
\[ \sum f_i \kappa_i^2 \geq f_n \kappa_n^2 \geq \frac{\theta^2}{n} \kappa_1^2 \sum f_i, \]
we obtain an upper bound for $(v^2 - u^2 - c) \kappa_1$.

\vspace{2mm}

{\bf Case (ii)}. Assume that $\kappa_n > - \theta \kappa_1$.
Denote
\[ I = \{  i \,| \, f_1 \geq \theta^2 f_i  \},\quad  J = \{  i \,| \, f_1 < \theta^2 f_i  \}. \]
By \eqref{eq2-16} and Cauchy-Schwartz inequality, we have
\begin{equation} \label{eq2-26}
\begin{aligned}
 & \sum f_i \frac{(v^2 - u ^2)_i^2}{(v^2 - u^2 - c)^2} \leq  \sum_{j \in I} f_j \frac{(v^2 - u ^2)_j^2}{(v^2 - u^2 - c)^2} \\
 & + \frac{2}{b^2} \sum_{j \in J} f_j \Big( \frac{(\nu^{n + 1})_j^2}{(\nu^{n + 1} - a)^2} + \frac{2 h_{11j}^2}{h_{11}^2} + 2 \beta^2 \frac{u_j^2}{u^4} \Big),
\end{aligned}
\end{equation}
\begin{equation} \label{eq2-27}
\begin{aligned}
& \frac{F^{ii} h_{11i}^2}{h_{11}^2} \leq (1 + \delta_2) \sum_{j \in I} f_j \frac{(\nu^{n + 1})_j^2}{(\nu^{n + 1} - a)^2}  + 2 (1 + \delta_2^{- 1}) \beta^2 \sum_{j \in I} f_j \frac{u_j^2}{u^4} \\
& + 2 (1 + \delta_2^{- 1}) b^2 \sum_{j \in I} f_j \frac{(v^2 - u^2)_j^2}{(v^2 - u^2 - c)^2} + \sum_{j \in J} \frac{F^{jj} h_{11j}^2}{h_{11}^2},
\end{aligned}
\end{equation}
and
\begin{equation} \label{eq2-23}
\begin{aligned}
& - \frac{4 b f_i u_i (X \cdot \tau_i)}{u (v^2 - u^2 - c)} = \frac{2 b f_i u_i (v^2 - u^2)_i}{u (v^2 - u^2 - c)} \\
\geq &  \sum_{j \in I} \frac{2 b f_j u_j (v^2 - u^2)_j}{u (v^2 - u^2 - c)}  + \sum_{j \in J} \frac{2 f_j u_j}{u} \Big( - \frac{h_{11j}}{h_{11}} + \frac{(\nu^{n + 1})_j}{\nu^{n + 1} - a} \Big) \\
\geq & - \sum_{j \in I} f_j - b^2 \sum_{j \in I} f_j \frac{ (v^2 - u^2)_j^2}{(v^2 - u^2 - c)^2} \\
& - \delta_2^{-1} \sum_{j \in J} f_j - 2 \delta_2 \sum_{j \in J} f_j \Big( \frac{h_{11j}^2}{h_{11}^2} + \frac{(\nu^{n + 1})_j^2}{(\nu^{n + 1} - a)^2} \Big).
\end{aligned}
\end{equation}
Taking \eqref{eq2-26}--\eqref{eq2-23} into \eqref{eq2-18-1}, we obtain
\begin{equation} \label{eq2-28}
\begin{aligned}
& \Big( \kappa_1 - \frac{1 + (\nu^{n + 1})^2}{\nu^{n + 1} - a} - \frac{\beta \nu^{n + 1}}{u} - \frac{2 b (X \cdot \nu) u}{v^2 - u^2 - c} \Big) \psi - C \\
 & + \Big( \frac{\beta}{u} + \frac{a - 2 - 4 a^{- 1}}{\nu^{n+1} - a} - 1 - \delta_2^{-1} \Big) \sum f_i
  + \frac{3 a}{4 (\nu^{n+1} - a)}  \sum f_i \kappa_i^2     \\
&    - \Big( b + b^2 + 2 (1 + \delta_2^{- 1}) b^2 \Big) \sum_{j \in I} f_j \frac{(v^2 - u^2)_j^2}{(v^2 - u^2 - c)^2} \\
& - ( 2 b^{- 1}  + 3 \delta_2 ) \frac{F^{ii} (\nu^{n+1})_i^2}{(\nu^{n+1} - a)^2}
  - \frac{4 \beta^2}{b} \sum_{j \in J} f_j \frac{u_j^2}{u^4}  - 2 (1 + \delta_2^{- 1}) \beta^2 \sum_{j \in I} f_j \frac{u_j^2}{u^4} \\
& - \frac{F^{ij, rs} h_{ij1} h_{rs1}}{\kappa_1}  - ( 1 + 4 b^{- 1} + 2 \delta_2 ) \sum_{j \in J} \frac{f_j h_{11j}^2}{h_{11}^2}   \leq 0.
\end{aligned}
\end{equation}
Applying \eqref{eq2-31} and choosing $\delta_2  = \frac{a^2}{16}$, \eqref{eq2-28} becomes
\begin{equation} \label{eq2-32}
\begin{aligned}
& \Big( \kappa_1 - \frac{1 + (\nu^{n + 1})^2}{\nu^{n + 1} - a} - \frac{\beta \nu^{n + 1}}{u} - \frac{2 b (X \cdot \nu) u}{v^2 - u^2 - c} \Big) \psi - C \\
 & + \Big( \frac{\beta}{u} + \frac{\frac{5 a}{8} - 2 - 4 a^{- 1}}{\nu^{n+1} - a} - 1 - \frac{16}{a^2} - \frac{4}{b (\nu^{n + 1} - a)^2} - \frac{4 \beta^2}{b u^2} \Big) \sum f_i
 \\
 &  + \Big( \frac{ 3 a }{8 (\nu^{n+1} - a)} - \frac{4}{b (\nu^{n + 1} - a)^2} \Big) \sum f_i \kappa_i^2 - \frac{ C \big( b +  (3 + 32 a^{- 2} ) b^2 \big)}{\theta^2 (v^2 - u^2 - c)^2} f_1   \\
&    - \frac{2 (1 + 16 a^{- 2}) \beta^2}{\theta^2 u^2} f_1
 - \frac{F^{ij, rs} h_{ij1} h_{rs1}}{\kappa_1}  - \Big( 1 + \frac{4}{b} + \frac{a^2}{8}  \Big) \sum_{j \in J} \frac{f_j h_{11j}^2}{h_{11}^2}   \leq 0.
\end{aligned}
\end{equation}
By an inequality of Andrews \cite{And} and Gerhardt \cite{Ger}, we have
\[ \begin{aligned}
& - \frac{F^{ij, rs} h_{ij1} h_{rs1}}{\kappa_1}
\geq   \frac{2}{\kappa_1} \sum_{j \in J} \frac{f_j - f_1}{\kappa_1 - \kappa_j} h_{11j}^2 \geq 2 (1 - \theta) \sum_{j \in J} \frac{f_j h_{11j}^2}{h_{11}^2}.
\end{aligned} \]
Choosing $\theta = \frac{1}{4}$, we have
\[
 - \frac{F^{ij, rs} h_{ij1} h_{rs1}}{\kappa_1}  - \Big( 1 + \frac{4}{b} + \frac{a^2}{8}  \Big) \sum_{j \in J} \frac{f_j h_{11j}^2}{h_{11}^2} \\
\geq  \Big( \frac{1}{2} - \frac{4}{b} - \frac{1}{8}  \Big) \sum_{j \in J} \frac{f_j h_{11j}^2}{h_{11}^2}.
 \]
Also, choosing $\beta$ sufficiently large, \eqref{eq2-32} reduces to
\begin{equation*}
\begin{aligned}
& \Big( \kappa_1 - \frac{1 + (\nu^{n + 1})^2}{\nu^{n + 1} - a} - \frac{\beta \nu^{n + 1}}{u} - \frac{2 b (X \cdot \nu) u}{v^2 - u^2 - c} \Big) \psi - C \\
 & + \Big( \frac{\beta}{2 u} - \frac{4}{b (\nu^{n + 1} - a)^2}  - \frac{4 \beta^2}{b u^2} \Big) \sum f_i
   + \Big( \frac{ 3 a }{8 (\nu^{n+1} - a)} - \frac{ 4 }{b (\nu^{n + 1} - a)^2} \Big) \sum f_i \kappa_i^2 \\
&  - \frac{ C \big( b +  (3 + 32 a^{- 2} ) b^2 \big)}{(v^2 - u^2 - c)^2} f_1  - \frac{C (1 + 16 a^{- 2}) \beta^2}{u^2} f_1
  + \Big( \frac{3}{8} - \frac{4}{b} \Big) \sum_{j \in J} \frac{f_j h_{11j}^2}{h_{11}^2}  \leq 0.
\end{aligned}
\end{equation*}
Then choosing $b$ sufficiently large, we have
\begin{equation*}
\begin{aligned}
& \Big( \kappa_1 - \frac{1 + (\nu^{n + 1})^2}{\nu^{n + 1} - a} - \frac{\beta \nu^{n + 1}}{u} - \frac{2 b (X \cdot \nu) u}{v^2 - u^2 - c} \Big) \psi - C
 \\
 &  + \frac{ 3 a }{16 (\nu^{n+1} - a)}   f_1 \kappa_1^2
  - \frac{ C \big( b +  (3 + 32 a^{- 2} ) b^2 \big)}{ (v^2 - u^2 - c)^2} f_1  - \frac{C (1 + 16 a^{- 2}) \beta^2}{u^2} f_1 \leq 0.
\end{aligned}
\end{equation*}
We thus obtain an upper bound for $(v^2 - u^2 - c) \kappa_1$.

Now we have finished the proof of \eqref{eq2-7}. In view of \eqref{eq2-14} and \eqref{eq2-12}, we can deduce that
\begin{equation} \label{eq2-33}
\max\limits_{\overline{\Omega_{\epsilon_0}}} \big| \kappa_i [ u^{\epsilon}] \big| \,\leq \, C, \quad \forall \,\, i = 1, \ldots, n, \quad \forall \,\, 0 < \epsilon < \frac{\delta_{\epsilon_0}}{2}.
\end{equation}

\vspace{2mm}

\subsection{Uniqueness}~

\vspace{2mm}

Under assumption \eqref{eqn12}, we can prove the following uniqueness result.

\begin{thm}
Under assumption \eqref{eqn12}, the admissible solution $u \in C^{2}(\Omega) \cap C^0(\overline{\Omega})$  to  asymptotic Plateau problem \eqref{eqn9} is unique.
\end{thm}
\begin{proof}
Let $u$ and $v$ be respectively an admissible solution and subsolution to
\eqref{eqn9} in $\Omega$.
We claim that $u \geq v$ in $\Omega$.

Suppose not. Then $v - u$ achieves a positive maximum at $x_0 \in \Omega$. Consequently, we have
\[  v(x_0) > u(x_0),\quad D v(x_0) = D u(x_0), \quad D^2 v (x_0) \leq D^2 u(x_0). \]
Now we consider the deformation $u[s] = s v + (1 - s) u$ for $s \in [0, 1]$. By direct calculation, we can verify that $u[s]$ is admissible near $x_0$ for any $s \in [0, 1]$.
This is because at $x_0$,
\[ \begin{aligned}
   & \delta_{ij} + u[s]  {\gamma}^{ik}  ( u[s] )_{kl}  \gamma^{lj}
  \geq  \delta_{ij} + u[s] {\gamma}^{ik}  v_{kl} \gamma^{lj}  \\
   = & (1 - s) \Big( 1 - \frac{u}{v}\Big) \delta_{ij} + \frac{u[s]}{v} \Big( \delta_{ij} + v  \gamma^{ik} v_{kl} \gamma^{lj} \Big).
  \end{aligned}
\]

Note that equation \eqref{eqn9} can be rewritten as
\[ f (\kappa[u]) = F (A [u]) =  G (D^2 u, D u, u) = G[u] = \psi (x, u). \]
Denote
\[ G^{ij} = \frac{\partial G}{\partial u_{ij}},  \quad G^i = \frac{\partial G}{\partial u_i}, \quad G_u = \frac{\partial G}{\partial u}. \]
Now we can define a differentiable function for $s \in [0, 1]$,
\[ a(s) = G ( D^2 u[s], D u[s], u[s] ) (x_0) - \psi ( x_0, u[s](x_0) ). \]
Since $a(0) = 0$ and $a(1) \geq 0$, there exists $s_0 \in [0, 1]$ such that $a(s_0) = 0$ and $a'(s_0) \geq 0$, or equivalently,
\[
 G \big[  u[s_0] \big] (x_0) = \psi (x_0, u[s_0](x_0) ), \]
and
\begin{equation} \label{eq2-30}
\begin{aligned}
& G^{ij} \big[  u[s_0] \big]  (x_0)  D_{ij}  (v - u)(x_0)
 + G^i \big[  u[s_0] \big] (x_0)  D_i  (v - u)(x_0)
 \\ & +  \big(G_u \big[  u[s_0] \big] (x_0) - \psi_u \big( x_0, u[s_0] (x_0) \big) \big)  (v - u)(x_0) \geq 0.
\end{aligned}
\end{equation}
Also note that
\[ G^{ij} = \frac{u}{ w} F^{kl} \gamma^{ki} \gamma^{j l}, \quad
G_u  = \frac{1}{u} \big( \psi - \frac{1}{w} \sum f_i \big). \]
Thus,
\[ G^{ij} \big[  u[s_0] \big]  (x_0)  D_{ij}  (v - u)(x_0) \leq 0 \]
and
\[
G_u \big[ u[s_0] \big](x_0) - \psi_u ( x_0, u[s_0] (x_0) ) < 0
\]
by assumption \eqref{eqn12}. But then \eqref{eq2-30} can not hold. Thus, we proved the claim. Consequently, we obtain the uniqueness.
\end{proof}


\vspace{4mm}

\begin{thebibliography}{9}

\vspace{4mm}


\bibitem{And}
B. Andrews,{ \em Contraction of convex hypersurfaces in Euclidean space}, Calc. Var. Partial Differential Equations {\bf 2} (1994), 151--171.




\bibitem{CNSV}
L. Caffarelli, L. Nirenberg and J. Spruck,{ \em Nonlinear second-order elliptic equations V. The Dirichlet problem for Weingarten hypersurfaces}, Comm. Pure Appl. Math. {\bf 41} (1988), 41--70.



\bibitem{Ger}
C. Gerhardt,{ \em Closed Weingarten hypersurfaces in Riemannian manifolds}, J. Differential Geom.{ \bf 43} (1996), 612--641.



\bibitem{GS00}
B. Guan and J. Spruck,{ \em Hypersurfaces of Constant Mean Curvature in Hyperbolic Space with Prescribed Asymptotic Boundary at Infinity}, Amer. J. Math. {\bf 122} (2000), 1039--1060.



\bibitem{GS10}
B. Guan and J. Spruck,{ \em Hypersurfaces of constant curvature in hyperbolic space II}, J. European Math. Soc. {\bf 12} (2010), 797--817.


\bibitem{GS11}
B. Guan and J. Spruck,{ \em Convex hypersurfaces of constant curvature in hyperbolic space}, Surveys in Geometric Analysis and Relativity ALM {\bf 20} (2011), 241--257.


\bibitem{GSS09}
B. Guan, J. Spruck and M. Szapiel,{ \em Hypersurfaces of constant curvature in hyperbolic space I}, J. Geom. Anal. {\bf 19} (2009), 772--795.



\bibitem{GSX14}
B. Guan, J. Spruck and L. Xiao,{ \em Interior curvature estimates and the asymptotic Plateau problem in hyperbolic space}, J. Differential Geom. {\bf 96} (2014), 201--222.


\bibitem{GQ17}
P. Guan and G. Qiu,{ \em Interior $C^2$ regularity of convex solutions to prescribing scalar curvature equations}, Duke Math. J. {\bf 168} (2019), 1641--1663.



\bibitem{Po78}
A. V. Pogorelov,{ \em The Minkowski multidimensional problem}, Wiley, New York, 1978.


\bibitem{Sheng-Urbas-Wang}
W. Sheng, J. Urbas and X.-J. Wang,{ \em Interior curvature bounds for a class of curvature equations}, Duke Math. J. {\bf 123} (2004), 235--264.



\bibitem{Sui2019}
Z. Sui,{ \em Convex hypersurfaces with prescribed scalar curvature and asymptotic boundary in hyperbolic space}, Calc. Var. Partial Differential Equations {\bf 60} (2021): 45.


\bibitem{Sui-Sun}
Z. Sui and W. Sun,
{ \em Lipschitz continuous hypersurfaces with prescribed curvature and asymptotic boundary in hyperbolic space},
International Mathematics Research Notices, rnab244, https://doi.org/10.1093/imrn/rnab244.



\bibitem{Tru-Urb84}
N. Trudinger and J. Urbas,{ \em On second derivative estimates for equations of Monge-Amp\`ere type}, Bulletin of the Australian Mathematical Society {\bf 30} (1984), 321--334.



\bibitem{Ur90}
J. Urbas,{ \em On the existence of nonclassical solutions for two classes of fully nonlinear elliptic equations}, Indiana Univ. Math. J. {\bf 39} (1990), 355--382.







\end{thebibliography}
\end{document}